\documentclass[10pt,a4paper]{article}
\usepackage{fullpage}
\usepackage{amsfonts,amsmath,amssymb}
\usepackage{amsthm}
\usepackage{graphicx}
\usepackage{sectsty}

\theoremstyle{plain}
\newtheorem{theorem}{Theorem}[section]
\newtheorem{lemma}[theorem]{Lemma}

\theoremstyle{definition}
\newtheorem{definition}[theorem]{Definition}
\theoremstyle{remark}
\newtheorem{remark}[theorem]{Remark}
\numberwithin{equation}{section}

\sectionfont{\large}
\newenvironment{acknowledgement}[1][Acknowledgement
]{\begin{trivlist} \item[\hskip \labelsep {\bfseries
#1}]}{\end{trivlist}}

\begin{document}
\title{Some Inverse Spectral Results in  Exterior Transmission Problem}
         
\author{Lung-Hui Chen$^1$}\maketitle\footnotetext[1]{Department of
Mathematics, National Chung Cheng University, 168 University Rd.
Min-Hsiung, Chia-Yi County 621, Taiwan. Email:
mr.lunghuichen@gmail.com;\,lhchen@math.ccu.edu.tw. Fax:
886-5-2720497.}
\begin{abstract}
We consider an inverse spectral theory in a domain with the cavity that is bounded by a penetrable inhomogeneous medium. An ODE system is constructed piecewise through the solutions inside and outside the cavity. The ODE system is connected to the PDE system via the analytic continuation. For each scattered angle, we describe its eigenvalue density in the complex plane, and prove an inverse uniqueness on the inhomogeneity by the measurements in the far-fields.
\end{abstract}
\section{Introduction and Preliminaries}
Let us consider the following scattering theory.
\begin{eqnarray}\label{1}
\left\{%
\begin{array}{ll}
\Delta u(x)+k^2n(x)u(x)=0,\,x\in\mathbb{R}^3;\vspace{4pt}\\\vspace{4pt}
u(x)=u^i(x)+u^s(x),\,x\in\mathbb{R}^3\setminus D; \\
\lim_{|x|\rightarrow\infty}|x|\{\frac{\partial u^s(x)}{\partial |x|}-iku^s(x)\}=0,
\end{array}%
\right.
\end{eqnarray}
where
\begin{eqnarray*}
&&u(x)\mbox{ is  the total wave};\nonumber\\
&&u^s(x)\mbox{ is  the scattered wave};\nonumber\\
&&u^i(x):=e^{ikx\cdot d},\,k\in\mathbb{\mathbb{C}},\,x\in\mathbb{R}^3,\, d\in\mathbb{S}^2,\mbox{which is the incident wave}.
\end{eqnarray*}
The  problem occurs when the plane waves are perturbed by the inhomogeneity specified by the index of refraction $n(x)$. The inverse problem is to recover the information on the index of refraction $n(x)$  by the measurements of the scattered wave-fields in the far-fields. The problem is common in many disciplines of science and technology such as sonar and radar, geophysical sciences, astrophysics, and non-destructive testing in instrument manufacturing.  

Out of the numerical motivation in their research in inverse scattering theory, Kirsch \cite{Kirsch86}, and Colton and Monk \cite{Colton} reduce the problem~(\ref{1}) into the following class of inverse spectral problem.
\begin{eqnarray}\label{1.1}
\left\{%
\begin{array}{ll}
    \Delta w+k^2n(x)w=0  & \hbox{ in }D; \vspace{3pt}\\
    \Delta v+k^2v=0& \hbox{ in }D; \vspace{3pt}\\
    w=v & \hbox{ on }\partial D; \vspace{3pt}\\
    \frac{\partial w}{\partial \nu}=\frac{\partial v}{\partial \nu}& \hbox{ on }\partial D,
    \end{array}%
\right.
\end{eqnarray}
where
$\nu$ is the unit outer normal. In this paper, we assume that $D$ is a starlike domain in $\mathbb{R}^3$ containing the origin with the boundary $\partial D$, and that ${\rm supp}(1-n)$ is outside $D$ that is  contained in some bounded domain $D'$.
The inhomogeneity $n\in\mathcal{C}^2(\mathbb{R}^3)$ and
$n(x)>0$ for all $x\in \mathbb{R}^3$, and the Laplacian in this paper is given by
\begin{equation}\label{L}
\Delta=\frac{1}{r^2}\frac{\partial}{\partial r}r^2\frac{\partial}{\partial r}+\frac{1}{r^2\sin{\varphi}}\frac{\partial}{\partial \varphi}\sin\varphi\frac{\partial}{\partial \varphi}
+\frac{1}{r^2\sin^2{\varphi}}\frac{\partial^2}{\partial \theta^2}.
\end{equation}

\par
Let us assume the boundary  $\partial D$  is defined by
\begin{eqnarray}
R=R(\hat{x})\in \mathcal{C}^1(\mathbb{S}^2;\mathbb{R}^+),
\end{eqnarray}
where $\mathbb{S}^2$ is the unit sphere and $\hat{x}:=(\theta,\varphi)$ is the spherical coordinate, and $r:=|x|$. 
The equation~(\ref{1.1}) is called the homogeneous exterior transmission eigenvalue problem \cite{Colton2,Colton6,Colton7}. We say $k$ is an exterior transmission eigenvalue if and only if it parameters a non-trivial eigenfunction pair of~(\ref{1.1}). 
\par

The exterior transmission problem happens naturally when  the plane waves are perturbed in the exterior of the cavity $D$ surrounded by certain inhomogeneity. The free wave fields are generated in the cavity, and propagate through the inhomogeneity defined the index of refraction to the far-fields. The inverse problem is to find the index of refraction by the measurements in the far-fields. We refer the scattering and inverse scattering theory of this problem to \cite{Aktosun,Cakoni3,Colton2,Colton6,Colton7,G}. 
To ensure the well-posedness of the scattered wave fields, we impose the Sommerfeld radiation conditions to~(\ref{1.1}).
\begin{eqnarray*}
&&\lim_{r\rightarrow\infty}r\{ \frac{\partial w}{\partial r}-i k w\}=0;\vspace{3pt}\\
&&\lim_{r\rightarrow\infty}r\{ \frac{\partial v}{\partial r}-i k v\}=0,
\end{eqnarray*}
which is typical in scattering theory \cite{Colton2,Isakov}.
\par
Let us expand the solution $(w,v)$ of~(\ref{1.1}) in two series of spherical harmonics by Rellich theory \cite[p.\,32, p.\,227]{Colton2}. This is a classic result holds for the Helmholtz equation outside a sphere. Here we choose the sphere large enough such that it contains the perturbation $n$. Then the following asymptotic identities hold.
\begin{eqnarray}\label{1.3}
\left\{
  \begin{array}{ll}v(x;k)=\sum_{l=0}^{\infty}\sum_{m=-l}^{m=l}a_{l,m}j_l(k r)Y_l^m(\hat{x});\vspace{6pt}\\
w(x;k)=\sum_{l=0}^{\infty}\sum_{m=-l}^{m=l}b_{l,m}\frac{ y_l(r)}{r}Y_l^m(\hat{x}),
\end{array}
\right.
\end{eqnarray}
where $r:=|x|$, $R_0\leq r<\infty;$ $\hat{x}=(\theta,\varphi)\in\mathbb{S}^2$;
$j_l$ is the spherical Bessel function of first kind of order $l$. The summations converge uniformly and absolutely on the compact subsets of $|x|=r\geq R_0$, with a sufficiently large $R_0$ containing $D'$.

\par
The spherical harmonics
\begin{equation}\label{S}
Y_l^m(\theta,\varphi):=\sqrt{\frac{2l+1}{4\pi}\frac{(l-|m|)!}{(l+|m|)!}}P_l^{|m|}(\cos\theta)e^{im\varphi},
\,m=-l,\ldots,l;\,l=0,1,2,\ldots,
\end{equation}
is a complete orthonormal system in $L^2(\mathbb{S}^2)$, and
\begin{equation}\nonumber
P_l^m(t):=(1-t^2)^{m/2}\frac{d^mP_l(t)}{dt^m},\,m=0,1,\ldots,l,
\end{equation}
where the Legendre polynomials $P_l$, $l=0,1,\ldots,$ give a complete orthogonal system in $L^2[-1,1]$.
We refer the details on the spherical harmonics and its applications to geometry to Groemer's book \cite{Groemer}.

\par
According to the orthogonality of the spherical harmonics, the functions
\begin{eqnarray}\label{18}
\left\{
  \begin{array}{ll}
&v_{l,m}(x;k):=a_{l,m}j_l(k r)Y_l^m(\hat{x});\vspace{6pt}\\
&w_{l,m}(x;k):=\frac{b_{l,m}y_l(r)}{r}Y_l^m(\hat{x})
 \end{array}
\right.
\end{eqnarray}
satisfy the first two equations of~(\ref{1.1}) independently on the compact subsets suitably away from $D'\setminus D$.

\par
Given one \textbf{fixed} incident $\hat{x}\in\mathbb{S}^2$, we can rotate the geometry and the perturbation on the $\hat{x}$ around the origin. Accordingly, we can extend uniquely the series $\{v_{l,m}(x;k)\}$ and $\{w_{l,m}(x;k)\}$ into $|x|\leq R_0$ along that fixed incident $\hat{x}\in\mathbb{S}^2$ by applying the Laplacian~(\ref{L}). For each $w_{l,m}(x;k)$ solving the Helmholtz equation, the Fourier coefficient $y_{l}(r;k)$ is equivalent to satisfy the following ODE:
\begin{eqnarray}\label{1.7}
\left\{
  \begin{array}{ll}
    y_l''+(k^2n(r\hat{x})-\frac{l(l+1)}{r^2})y_l=0,\,0<r<\infty;\vspace{5pt}\\
    \underset{r\rightarrow0^{+}}{\lim}\{\frac{y_l(r)}{r}-j_l(k r)\}=0.
  \end{array}
\right.
\end{eqnarray}
The behavior of the Bessel function $j_l(kr)$ near $r=0$ is found in \cite[p.\,437]{Ab}. We refer the initial condition~(\ref{1.7}) to \cite{Mc}.

\par
Surely, $y_l(r;k)$ \textbf{depends} on the incident $\hat{x}$ in $|x|\leq R_0$. We denote the solution of~(\ref{1.7}) as $\hat{y}_l(r;k)$, which corresponds to some eigenfunction of~(\ref{1.1}). We will explain the correspondence between the spectrum of~(\ref{1.7}) and~(\ref{1.1}) in Lemma \ref{41}, Lemma \ref{3388}, and Lemma \ref{33}.
At least, the assumption of~(\ref{1.1}) implies that if there is an eigenvalue $k$ of~(\ref{1.1}), then on $|x|=R_0$ we have
\begin{eqnarray}
\left\{\begin{array}{ll}
a_{l,m}j_l(k r) |_{r=R_0}=\frac{b_{l,m}\hat{y}_l(r)}{r}|_{r=R_0};\label{1.10}\vspace{5pt}\\
a_{l,m}\partial_r j_l(k r)|_{r=R_0} =\partial_r\frac{b_{l,m} \hat{y}_l(r)}{r}|_{r=R_0},
\end{array}
\right.
\end{eqnarray}
which is independently of $m$ and $\hat{x}\in\mathbb{S}^2$. In this paper, we study the zero set of~(\ref{1.10}).
Without loss of generality, we take  $$a_{l,m}=b_{l,m}=1$$ by applying the Sommerfield radiation condition to $w_{l,m}(x;k)$ and $v_{l,m}(x;k)$ respectively. 
Coordinate-wise, now we are looking for any $k\in\mathbb{C}$ such that $\frac{y_l(r;k)}{r}=j_{l}(k r)$ outside $D'\setminus D$ and inside $D$, which is an algebraic identity in $\mathbb{C}$. 

We define
\begin{equation}
\hat{n}(r):=n(r\hat{x}).
\end{equation}

\par
 For $-l\leq m\leq l$, $l=0,1,2,\ldots$, the existence of the non-zero constants in~(\ref{1.10}) is reduced to finding the zeros of
\begin{equation}\label{D}
\hat{D}_{l}(k;r=R_0):=\det\left(%
\begin{array}{cc}
  j_l(k r)|_{r=R_0}  & -\frac{\hat{y}_l(r)}{r}|_{r=R_0}\vspace{6pt}\\
  \{j_l(k r)\}'|_{r=R_0}& -\{\frac{\hat{y}_l(r)}{r}\}'|_{r=R_0}
\end{array}%
\right).
\end{equation}
If $\hat{y}_l(r;k_0)$ solves~(\ref{1.7}) and~(\ref{1.10}), then $\hat{y}_l(r;k_0)$ solves~(\ref{1.7}) and $\hat{D}_{l}(k_0)=0$, which is an algebraic constraint, and thus the theory on the zeros of the entire function theory plays a role.

We state the following inverse spectral theorem of~(\ref{1.1}).
\begin{theorem}\label{14}
Let $n^j$ be an unknown inhomogeneity to the background index of refraction $1$ in~(\ref{1.1}), $j=1,2$. If $n^1$ and $n^2$ have the same set of eigenvalues of~(\ref{1.1}) in $\mathbb{C}$, then $n^1\equiv n^2$.
\end{theorem}

\section{Asymptotic Solutions of ODE}
Let us consider the ODE with the
Liouville transformation \cite{Carlson2,Carlson3,Colton2,Po} for some fixed $\hat{x}$:
\begin{eqnarray}\nonumber
&z_{l}(\xi):=[n(r\hat{x})]^{\frac{1}{4}}y_{l}(r;k),
\end{eqnarray}
where
\begin{equation}
\xi(r)=  \int_0^r[n(\rho\hat{x})]^{\frac{1}{2}}d\rho.
\end{equation}
Therefore,
\begin{eqnarray}\label{2.1}
z_l''+[k^2-q(\xi)-\frac{l(l+1)}{\xi^2}]z_l=0,
\end{eqnarray}
in which
\begin{eqnarray}
q(\xi):=\frac{n''(r\hat{x})}{4[n(r\hat{x})]^2}-\frac{5}{16}\frac{[n'(r\hat{x})]^2}{[n(r\hat{x})]^3}+\frac{l(l+1)}{r^2n(r\hat{x})}-\frac{l(l+1)}{\xi^2}.
\end{eqnarray}
Let us drop the superscript on $\hat{x}$ for notation simplicity if the context is clear. The general solution of~(\ref{2.1})
has two independent fundamental solutions.
Let us apply the results from \cite[Lemma\,3.3]{Carlson2}, and consider $z_l(\xi;k)$ solving the following ODE.
\begin{eqnarray}\label{2.3}
\left\{
\begin{array}{ll}
-z_l''(\xi)+\frac{l(l+1)z_l(\xi)}{\xi^2}+q(\xi)z_l(\xi)=k^2z_l(\xi); \vspace{5pt}\\
 z_l(R;k)=-b;\,z_l'(R;k)=a,\,a,\,b\in\mathbb{R},
\end{array}
\right.
\end{eqnarray}
where the function $q(\xi)$ is assumed to be real-valued and square-integrable and $l\geq-1/2$.
The following estimate holds for $0\leq\xi\leq R$.
\begin{eqnarray}\label{2.8}
|z_l(\xi;k)+b\cos{k(R-\xi)}+a\frac{\sin{k(R-\xi)}}{k}|\leq \frac{K(\xi)}{|k|}\exp\{|\Im k|(R-\xi)\},\,|k|\geq1,
\end{eqnarray}
where
\begin{equation}\nonumber
K(\xi)\leq\exp\{\int_\xi^R\frac{|l(l+1)|}{t^2}+|q(t)|dt\}.
\end{equation}
We note here that the ODE~(\ref{2.3}) starts at $\xi=R$ and moves to the origin while \cite[Lemma\,3.3]{Carlson2} starts at $1$, and then moves toward the origin. We make it a \textbf{two-way} construction of solutions, which is the most important ingredient of this paper. For the ODE starting at $\xi=R$, that is if and only r=R, and moving to the infinity, we have
\begin{eqnarray}\label{2.9}
|z_l(\xi;k)+b\cos{k(\xi-R)}-a\frac{\sin{k(\xi-R)}}{k}|\leq \frac{\tilde{K}(\xi)}{|k|}\exp\{|\Im k|(\xi-R)\},\,|k|\geq1,
\end{eqnarray}
where
\begin{equation}
\tilde{K}(\xi)\leq\exp\{\int_R^\xi\frac{|l(l+1)|}{t^2}+|q(t)|dt\}.
\end{equation}

\par
For the application in this paper, we combine the estimates of the solution of~(\ref{2.3}) by considering the initial condition $\hat{D}_{l}(R;k)=0$ for $r\geq R$, $R=R(\hat{x})$, which is equivalent to the following algebraic system.
\begin{eqnarray}\label{266}
&&\hat{y}_l(R;k)=Rj_l(R;k);\vspace{4pt}\\
&&\hat{y}_l'(R;k)=j_l(R k)+R k j_l'(R k),\,k\in\mathbb{C},\label{277}
\end{eqnarray}
which is the transmission condition of~(\ref{1.1}) on $\partial D$. 
Most importantly, the general solution of~(\ref{2.3}) for $r\geq R$ is spanned by two of its fundamental solutions as in the following lemma.
\begin{lemma}\label{21}
For $k$ near real axis, the following asymptotics holds.
\begin{equation}\label{214}
\hat{y}_l(r;k)=[j_l(R k)+R kj_l'(Rk)]\frac{\sin\{k[\xi(r)-R]\}}{k}+Rj_l(Rk)\cos\{k[\xi(r)-R]\}+O(\frac{1}{k}),\,r\geq R.
\end{equation}
Particularly,  $\hat{y}_l(r;k)$ is bounded in $0i+\mathbb{R}$.
\end{lemma}
\begin{proof}
~(\ref{214}) follows from the general theory of ODE and~(\ref{2.9}) if we are required by the initial condition~(\ref{266}) and~(\ref{277}). All functions in~(\ref{214}) are bounded in $0i+\mathbb{R}$, because of~(\ref{2.9}).

\end{proof}
The analysis is reversible into the domain $D$ by considering~(2.8) with initial condition~(\ref{266}) and~(\ref{277}).

\section{Poly\'{a}-Cartwright-Levinson Theory}
We collect a few facts from entire function theory
\cite{Cartwright2,Koosis,Levin,Levin2}.
\begin{definition}
Let $f(z)$ be an integral function of order $\rho$, and let
$N(f,\alpha,\beta,r)$ denote the number of the zeros of $f(z)$
inside the angle $[\alpha,\beta]$ and $|z|\leq r$. We define the
density function as
\begin{equation}\label{Den}
\Delta_f(\alpha,\beta):=\lim_{r\rightarrow\infty}\frac{N(f,\alpha,\beta,r)}{r^{\rho}},
\end{equation}
and
\begin{equation}
\Delta_f(\beta):=\Delta_f(\alpha_0,\beta),
\end{equation}
with some fixed $\alpha_0\notin E$ such that $E$ is at most a
countable set \cite{Boas,Cartwright2,Koosis,Levin,Levin2}.
\end{definition}
\begin{definition}
Let $f(z)$ be an integral function of finite order $\rho$ in the
angle $[\theta_1,\theta_2]$. We call the following quantity as the
indicator function of the function $f(z)$.
\begin{equation}\label{33333}
h_f(\theta):=\lim_{r\rightarrow\infty}\frac{\ln|f(re^{i\theta})|}{r^{\rho}},
\,\theta_1\leq\theta\leq\theta_2.
\end{equation}
\end{definition}
\begin{lemma}\label{333}
Let $f$, $g$ be two entire functions. Then the following two
inequalities hold.
\begin{eqnarray}
&&h_{fg}(\theta)\leq h_{f}(\theta)+h_g(\theta),\mbox{ if one limit exists};\label{2115}\\
&&h_{f+g}(\theta)\leq\max_\theta\{h_f(\theta),h_g(\theta)\},\label{2.16}
\end{eqnarray}
where the equality in~(\ref{2115}) holds if one of the functions is of completely regular growth, and secondly the equality~(\ref{2.16}) holds if the indicator of the two summands are not equal at some $\theta_0$.
\end{lemma}
\begin{proof}
 We can find
the details in \cite{Levin}.
\end{proof}
\begin{definition}
The following quantity is called the width of the indicator
diagram of entire function $f$:
\begin{equation}\label{d}
d=h_f(\frac{\pi}{2})+h_f(-\frac{\pi}{2}).
\end{equation}
\end{definition}
The distribution on the zeros of entire function of exponential type is described
precisely in the following Cartwright's theorem
\cite{Cartwright2,Levin,Levin2}. The following
statements are from Levin \cite[Ch.5, Sec.4]{Levin}.
\begin{theorem}[Cartwright]\label{C}
If an entire function of exponential type satisfies one of the
following conditions:
\begin{equation}\nonumber
\mbox{ the integral
}\int_{-\infty}^\infty\frac{\ln^+|f(x)|}{1+x^2}dx\mbox{ exists}.
\end{equation}
\begin{equation}\nonumber
|f(x)|\mbox{ is bounded on the real axis}.
\end{equation}
Then
\begin{enumerate}
    \item $f(z)$ is of class A and of completely regular growth,
    and its indicator diagram is an interval on the imaginary
    axis;
\item all of the zeros of the function $f(z)$, except possibly
those of a set of zero density, lie inside arbitrarily small
angles $|\arg z|<\epsilon$ and $|\arg z-\pi|<\epsilon$, where the
density
\begin{equation}\label{3.9}
\Delta_f(-\epsilon,\epsilon)=\Delta_f(\pi-\epsilon,\pi+\epsilon)=\lim_{r\rightarrow\infty}
\frac{N(f,-\epsilon,\epsilon,r)}{r}
=\lim_{r\rightarrow\infty}\frac{N(f,\pi-\epsilon,\pi+\epsilon,r)}{r},
\end{equation}
is equal to $\frac{d}{2\pi}$, where $d$ is the width of the
indicator diagram in~(\ref{d}). Furthermore, the limit
$\delta=\lim_{r\rightarrow\infty}\delta(r)$ exists, where
$$
\delta(r):=\sum_{\{|a_k|<r\}}\frac{1}{a_k};
$$
\item moreover,
\begin{equation}\nonumber
\Delta_f(\epsilon,\pi-\epsilon)=\Delta_f(\pi+\epsilon,-\epsilon)=0;
\end{equation}
\item the function $f(z)$ can be represented in the form
\begin{equation}\nonumber
f(z)=cz^me^{i\kappa
z}\lim_{r\rightarrow\infty}\prod_{\{|a_k|<r\}}(1-\frac{z}{a_k}),
\end{equation}
where $c,m,\kappa$ are constants and $\kappa$ is real;
\item the indicator
function of $f$ is of the form
\begin{equation}
h_f(\theta)=\sigma|\sin\theta|.
\end{equation}
\end{enumerate}
\end{theorem}
We refer the last statement to Levin \cite[p. 126]{Levin2}.
\begin{lemma}\label{37}
We have the following indicator functions.
$$h_{{j}'_l(kR_0)}(\theta)=h_{{j}_l(kR_0)}(\theta)=|R_0\sin\theta|,\,\theta\in[0,2\pi].$$
\end{lemma}
\begin{proof}
The spherical Bessel functions ${j}_l(kR_0)$ and ${j}'_l(kR_0)$ behave asymptotically like $\frac{\sin R_0 k}{k}$ and $\cos R_0 k$ respectively by considering the analysis in~(\ref{2.3}). The analysis on the Bessel function is classic \cite{Ab}. We refer the computation on their indicator functions to Cartwright theory \cite{Boas,Cartwright2,Koosis,Levin,Levin2}. We have applied the technique in inverse problems \cite{Chen,Chen3,Chen5,Chen9,Colton6}.
\end{proof}
\begin{lemma}\label{36}
The following asymptotic identity holds.
\begin{equation}\label{3112}
h_{\hat{D}_l(k;R_0)}(\theta)=h_{{j}'_l(kR_0)}(\theta)+h_{\hat{y}_l(R_0;k)}(\theta),\,\theta\in[0,2\pi].
\end{equation}
\end{lemma}
\begin{proof}
We begin with~(\ref{D}).
\begin{eqnarray}\label{DD}
\hat{D}_l(k;R_0)
&=&-j_l(k R_0)\frac{{\hat{y}}'_l(R_0;k)}{R_0}+j_l(k R_0)\frac{\hat{y}_l(R_0;k)}{R_0^2}+k {j}'_l(k R_0)\frac{\hat{y}_l(R_0;k)}{R_0}\vspace{8pt}\\\nonumber
&=&\frac{k{j}'_l(k R_0)\hat{y}_l(R_0;k)}{R_0}\{1-\frac{1}{k}\frac{j_l(k R_0)}{{j}'_l(k R_0)}\frac{{\hat{y}}'_l(R_0;k)}{\hat{y}_l(R_0;k)}
+\frac{1}{k R_0}\frac{j_l(k R_0)}{{j}'_l(k R_0)}\}\vspace{8pt}\\
&=&\frac{k{j}'_l(k R_0)\hat{y}_l(r;k)}{R_0}\{\hat{\alpha}_l(k)+O(\frac{1}{k})\},\label{216}
\end{eqnarray}
in which
\begin{equation}\label{3.16}
\hat{\alpha}_l(k):=1-\frac{1}{k}\frac{j_l(k R_0)}{{j}'_l(k R_0)}\frac{{\hat{y}}'_l(R_0;k)}{\hat{y}_l(R_0;k)},
\end{equation}
where we see that non-zero $$\frac{j_l(k R_0)}{{j}'_l(k R_0)}=O(1)$$ and non-zero $$\frac{{\hat{y}}'_l(R_0;k)}{\hat{y}_l(R_0;k)}=O(k)$$ away from its poles. Moreover, Lemma \ref{21} implies that ${\hat{y}}'_l(R_0;k)$ and $k\hat{y}_l(R_0;k)$ are asymptotically periodic functions. They are bounded when suitably away from the real axis.  Thus,~(\ref{33333}) shows the Lindel\"{o}f's indicator function $h_{\hat{\alpha}_l}(\theta)=0$. We refer the step-by-step computation to \cite{Boas,Chen,Chen3,Chen5,Levin,Levin2}. However, Lindel\"{o}f's indicator function for~(\ref{216}) is
\begin{equation}\label{2.15}
h_{\hat{D}_l(k;R_0)}(\theta)=h_{{j}'_l(kR_0)}(\theta)+h_{\hat{y}_l(R_0;k)}(\theta),\,\theta\neq0.
\end{equation}
Here we use~(\ref{2.15}).
\par
If $\hat{\alpha}_l(k)\equiv0$, then we have the non-zero second term in~(\ref{DD}). The indicator function is calculated similarly, and thus~(\ref{3112}) is proven again.

\end{proof}
\begin{remark}
We do not assume the minimal support of the perturbation $\{1-n\}$ in our assumption to~(\ref{1.1}). The indicator function~(\ref{2.15}) do not give the density function of the spectra of~(\ref{1.1}) through the Cartwright-Levinson Theorem \ref{C}.  
\end{remark}
\begin{lemma}\label{38}
We have the following indicator functions for $\hat{y}_l(r;k)$ and $\hat{y}'_l(r;k)$ for $r\geq R$.
$$h_{{\hat{y}}'_l(r;k)}(\theta)=h_{\hat{y}_l(r;k)}(\theta)=|\xi(r)||\sin\theta|,\,\theta\in[0,2\pi].$$
\end{lemma}
\begin{proof}
We apply~(\ref{2.15}) to Lemma \ref{21}. The Liouville transform of~(\ref{2.1}) is
$$
\xi(r)=\int_0^r\sqrt{n(\rho\hat{x})}d\rho,\,r\geq R,
$$in which $n=1$ outside $D'\setminus D$ by assumption.

\end{proof}

We look for a zero set of certain entire function that contains the eigenvalues of~(\ref{1.1}). The sequence $\{\hat{D}_{l}(k;r=R_{0})\}_{{\hat{x}}}$ plays a role in this paper.
\begin{lemma}\label{41}
$k$ is an eigenvalue of~(\ref{1.1})  if and only if $k$ is zero of $\hat{D}_{l}(k;r=R_0)$ for some $l$ and some $\hat{x}\in \mathbb{S}^2$, where $R_0$ is sufficiently large and given in~(\ref{1.3}).
\end{lemma}
\begin{proof}
Let $k\in\mathbb{C}$ be an eigenvalue of~(\ref{1.1}). By Rellich theory, the expansion~(\ref{1.3}) holds uniquely.  Hence, $\hat{D}_{l}(k;r=R_0)=0$ holds for all and for some $l$ in particular. 

For the sufficient condition,  $\{v_{l,m}(x;k)\}$ and $\{w_{l,m}(x;k)\}$ in~(\ref{18}) are independent eigenfunctions of Helmholtz equation for each $(l,m)$ for $|x|\geq R_0$.  If $k_0$  solves $\widetilde{D}_{l}(k;R_0)=0$ for some $l$ and for some $\tilde{x}\in\mathbb{S}^2$, then  for this $k_0$ we deduce from~(\ref{1.10}) that
\begin{eqnarray}\label{331}
\left\{\begin{array}{ll}
j_l(k_0 r) |_{r=R_0}=\frac{\tilde{y}_l(r;k_0)}{r}|_{r=R_0};\vspace{5pt}\\
\partial_r j_l(k_0 r)|_{r=R_0} =\partial_r\frac{ \tilde{y}_l(r;k_0)}{r}|_{r=R_0},
\end{array}
\right.
\end{eqnarray}
which holds \textbf{independently} for all $\hat{x}\in\mathbb{S}^2$ by Rellich theory. Let us consider~(\ref{331}) as an initial condition that works for all $\hat{x}\in\mathbb{S}^2$.
By the uniqueness and existence of ODE~(\ref{2.3}) and~(\ref{331}), the given $k_0$ defines some coefficients $\hat{y}_l(r;k_0)$ in $\mathbb{R}^3$ constructed as in~(\ref{2.8}),~(\ref{2.9}), and~(\ref{214}) for all other $\hat{x}\in\mathbb{S}^2$.  The extension routes of the Fourier coefficients $\hat{y}_l(r;k_0)$ up to $r=0$ are illustrated in the Figure 1. Therefore,  there exists eigenfunction pair $w(x;k_0)=v(x;k_0)$ inside $D$, and $\hat{D}_{l}(k_0;r=R)=0$. Taking directional derivatives near $D$, $$ \frac{\partial w(x;k_0)}{\partial \nu}=\frac{\partial v(x;k_0)}{\partial \nu}\mbox{ on }\partial D.$$
That makes $(w,v)$ a pair of eigenfunctions of~(\ref{1.1}).

\end{proof}
\begin{figure}\begin{center}\includegraphics[scale=0.3]{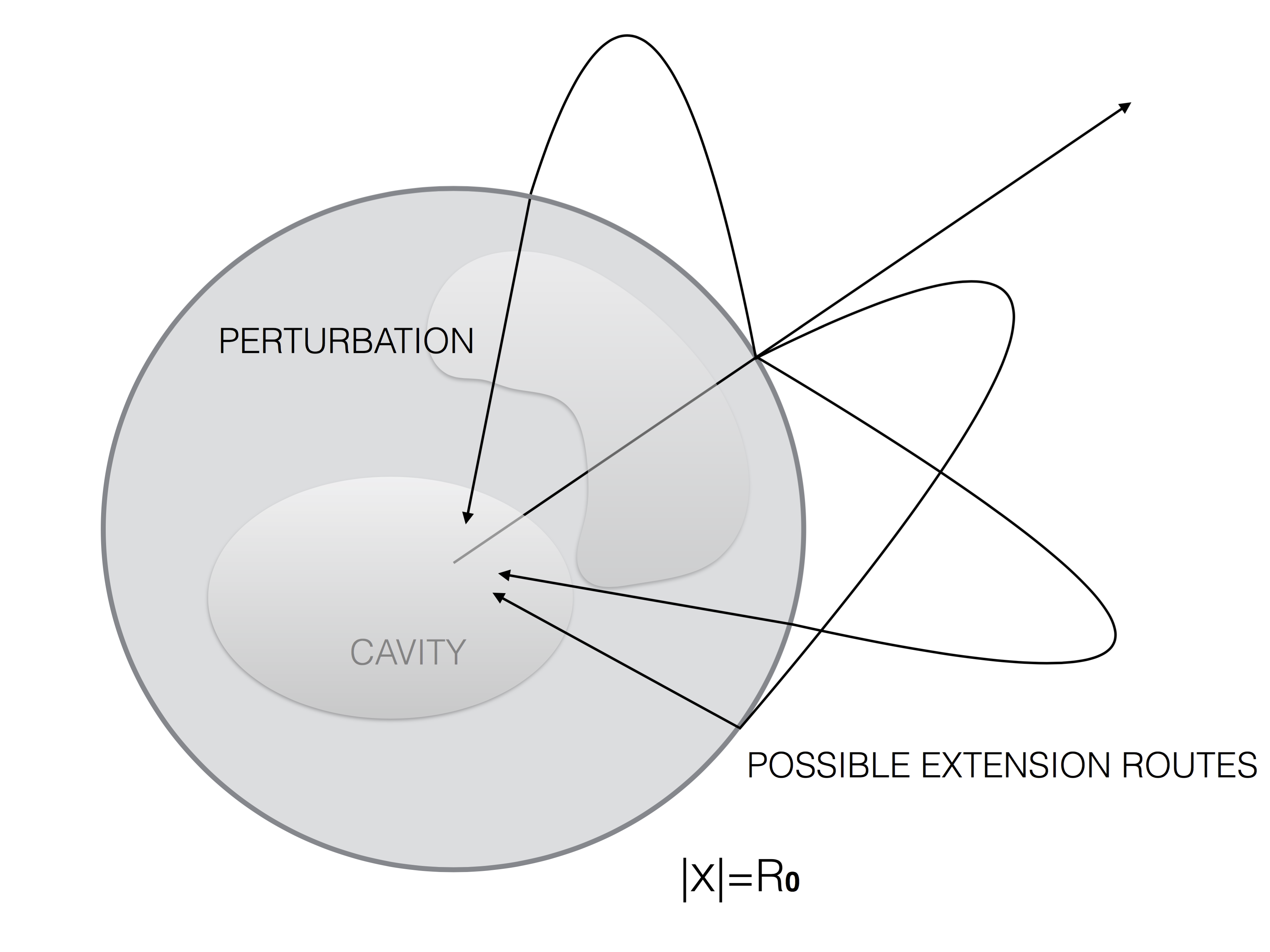}\caption{Rays of Extension Routes}\end{center}\end{figure}
The effective support of $\{1-n\}$ may not be minimal as the shown by the following lemma.
\begin{lemma}\label{3388}
Given a $\hat{x}\in\mathbb{S}^2$ and a fixed $k$, $\hat{D}_{l}(k;r)$ is locally constant near $r=R_0$ whenever $x\notin {\rm supp }\{1-n\}$.
\end{lemma}
\begin{proof}
Let  us add the initial condition~(\ref{1.10}) to $$\hat{y}_l''(r)+(k^2n(r\hat{x})-\frac{l(l+1)}{r^2})\hat{y}_l(r)=0.$$ The function $j_l(k r)$ and $ \frac{\hat{y}_l(r)}{r}$ satisfy the same ODE outside the perturbation. The lemma is proven by the uniqueness of ODE.

\end{proof}
\begin{lemma}\label{33}
$\hat{D}_{l}(k;r)=0$ for $x\in D$.
\end{lemma}
\begin{proof}
We have $w=v$ inside $D$. The uniqueness of Rellich's lemma~(\ref{1.3}) and the uniqueness of ODE imply that
$\frac{\hat{y}_l(r;k)}{r}=j_l(kr)$ for $r\leq R$. This proves the lemma.

\end{proof}

\section{A Proof of Theorem \ref{14}}
\begin{proof}
Let $n^1$ and $n^2$ be two indices of refraction with solutions $y_l^1(r;k)$ and $y_l^2(r;k)$ respectively with  the identical set of exterior eigenvalues $\mathcal{E}$, with the density given by the indicator function in Lemma \ref{36}  and~(\ref{3.9}) in Theorem \ref{C}. 
\par
By applying Lemma \ref{41} and~(\ref{1.10}), we have for each fixed $\hat{x}\in\mathbb{S}^{2}$,
\begin{eqnarray}\label{5.1}
&&\hat{y}_0^1(r;k)=\hat{y}_0^2(r;k);\vspace{8pt}\\
&&\partial_r\hat{y}_0^1(r;k)=\partial_r\hat{y}_0^2(r;k),\,k\in\mathcal{E},\,r= R_0.
\end{eqnarray}
Let $$F(k):=\hat{y}_0^1( R_0;k)-\hat{y}_0^2( R_0;k).$$ According to Lemma \ref{333}, we know that the indicator function
$$h_{F}(\theta)=\max\{h_{\hat{y}_0^1(R_0;k)}(\theta),h_{\hat{y}_0^2(R_0;k)}(\theta)\},$$
in which, by Lemma \ref{38}, we have
$$h_{\hat{y}_0^j(R_0;k)}(\theta)=|R+\int_{R}^{R_0}\sqrt{n^j(\rho\hat{x})}d\rho||\sin\theta|,\,j=1,2.$$

\par

We apply  Lemma \ref{36}, and Lemma \ref{38} to deduce that
$$h_{\hat{D}_0^j(k;R_0)}(\theta)>h_{\hat{y}_0^j(R_0;k)}(\theta),\,j=1,2,$$ and thus the exterior spectrum
$\mathcal{E}$ renders greater angle-wise density than the solution set of~(\ref{5.1}). This contradicts the maximal density of the zero set in Cartwright Theorem as stated
in~(\ref{3.9}).
Hence,
\begin{eqnarray}
&&\hat{y}_0^1(R_0;k)\equiv\hat{y}_0^2(R_0;k);\\
&&\partial_r\hat{y}_0^1(R_0;k)\equiv \partial_r\hat{y}_0^2(R_0;k).
\end{eqnarray}
Hence, $n^1$ and $n^2$ have the same set of norming constants and two independent spectra, Dirichlet and Neumann, to the following equation.
\begin{eqnarray}
\left\{
  \begin{array}{ll}
    \hat{y}_l''+(k^2n(r\hat{x})-\frac{l(l+1)}{r^2})\hat{y}_l=0,\,0<r<R_{0};\vspace{5pt}\\
    \underset{r\rightarrow0^{+}}{\lim}\{\frac{y_l(r)}{r}-j_l(k r)\}=0.
  \end{array}
\right.
\end{eqnarray}
By the inverse uniqueness of the Bessel operator \cite[Theorem\,1.2,\,Theorem\,1.3]{Carlson2}, we have $n^1(r\hat{x})\equiv n^2(r\hat{x})$ in $0\leq r\leq R_0$. The argument can be carried to all $\hat{x}\in\mathbb{S}^2$. This proves Theorem \ref{14}.

\end{proof}


\begin{acknowledgement}
The author wants to thank Prof. Shixu Meng for providing the manuscript \cite{Colton6,Colton7} in the preparation of this work, Prof. Chao-Mei Tu at NTNU for proofreading an earlier version of this manuscript, and  anonymous referees for suggesting some references in value distribution theory.
\end{acknowledgement}


\begin{thebibliography}{widest-label}
\bibitem{Aktosun}T. Aktosun, D. Gintides and V. G. Papanicolaou,
The uniqueness in the inverse problem for transmission eigenvalues
for the spherically symmetric variable-speed wave equation,
Inverse Problems, V. 27, 115004 (2011).
\bibitem{Ab}M. Abramowitz and I. A. Stegun (editors), Handbook of Mathematical Functions,      National Bureau of Standards   Applied Mathematics Series, 55, Washington D. C., 1956.
\bibitem{Boas}R.P. Boas, Entire functions, Academic Press, New York, 1954.
\bibitem{Cakoni3}F. Cakoni, D. Colton and S. Meng, The inverse scattering problem for a penetrable cavity with internal measurements, AMS, Contemporary Mathematics, 615, 71-88 (2014).
\bibitem{Carlson2}R. Carlson, A Borg-Levinson theorem for Bessel operators, Pacific Journal of Mathematics, Vol. 177, No. 1, 1--26 (1997).
\bibitem{Carlson3}R. Carlson, Inverse Sturm-Liouville problems with a singularity at zero, Inverse Problems, 10, 851--864 (1994).
\bibitem{Cartwright2}M. L. Cartwright, Integral Functions,
Cambridge University Press, Cambridge, 1956.
\bibitem{Chen}L. -H. Chen, An uniqueness result with some density theorems with interior transmission eigenvalues,  Applicable Analysis, DOI:
    10.1080/00036811.2014.936403.
\bibitem{Chen3}L. -H. Chen, A uniqueness theorem on the eigenvalues of spherically symmetric interior transmission
problem in absorbing medium, Complex Variables and Elliptic Equations, DOI:
    10.1080/17476933.2014.900055, 2014.
\bibitem{Chen5}L. -H. Chen,
On the inverse spectral theory in a non-homogeneous interior transmission problem, Complex Var. Elliptic Equ, DOI:  10.1080/17476933.2014.970541.
\bibitem{Chen9}L. -H. Chen, 
An inverse uniqueness in interior transmission
problem and its eigenvalue tunneling in  simple domain, Forthcoming in ''Adv. Math. Pays.''
\bibitem{Colton} D. Colton and P. Monk, The inverse scattering problem for time-harmonic acoustic waves in an inhomogeneous medium, Q. Jl. Mech. appl. Math. Vol. 41, 97--125 (1988).
\bibitem{Colton2}D. Colton and
 R. Kress, Inverse Acoustic and Electromagnetic Scattering Theory,
2rd ed. Applied Mathematical Science, V. 93, Springer--Verlag, 2013.

\bibitem{Colton6}D. Colton, Y.-J. Leung and S. Meng, The inverse spectral problem for exterior transmission problems, Inverse Problems, 30, 055010 (2014).
\bibitem{Colton7}D. Colton and S. Meng, Spectral properties of the exterior transmission eigenvalue problem, Inverse Problems, 30, 105010 (2014).

\bibitem{Groemer}H. Groemer, Geometric Applications of Fourier Series and Spherical Harmonics, Cambridge University Press, New York, 1996.
\bibitem{Isakov}V. Isakov, Inverse problems for partial differential equations, Applied Mathematical Sciences, V. 127, Springer--Verlag, New York, 1998.
\bibitem{Kirsch86}A. Kirsch, The denseness of the far field patterns for the transmission problem, IMA J. Appl. Math, 37, no. 3, 213--225 (1986).
\bibitem{Koosis}P. Koosis, The Logarithmic Integral I, Cambridge University Press, New York, 1997.

\bibitem{Levin}B. Ja. Levin, Distribution of zeros of entire
functions, revised edition, Translations of Mathematical
Monographs, American Mathematical Society, Providence, 1972.

\bibitem{Levin2}B. Ja. Levin, Lectures on entire functions,
Translation of Mathematical Monographs, V. 150, AMS, Providence, 1996.
\bibitem{G} G. Hu, J. Li, and H. Liu, Uniqueness in determining refractive indices by formally determined far-field data. Appl. Anal, 94,  no. 6, 1259--1269 (2015).
\bibitem{Mc}J. R. McLaughlin and P. L. Polyakov, On the uniqueness
of a spherically symmetric speed of sound from transmission
eigenvalues, Jour. Differential Equations, 107, 351--382 (1994).

\bibitem{Po}J. P\"{o}schel and E. Trubowitz, Inverse Spectral Theory,
Academic Press, Orlando, 1987.

\end{thebibliography}
\end{document}